\documentclass[11pt]{amsart}

\usepackage[T1]{fontenc}
\usepackage[latin1]{inputenc}
\usepackage[english]{babel}
\usepackage{amsmath,amssymb,amsthm} 
\usepackage{enumitem}
\usepackage{mymacros}
\renewcommand{\MR}[1]{}
\usepackage{color}
\newcommand{\tr}{\operatorname{tr}}
\newcommand{\Rat}{\operatorname{Rat}}
\usepackage{hyperref}

\usepackage[margin=1.5in]{geometry}

\newcommand{\mcc}{M\textsuperscript{c}Carthy}

\title[Finite dimensional dilations]{Dilation theory in finite dimensions and matrix convexity}
\author{Michael Hartz}
\address{Fachrichtung Mathematik, Universit\"at des Saarlandes, 66123 Saarbr\"ucken, Germany}
\email{hartz@math.uni-sb.de}
\thanks{M.H. was partially supported by a Feodor Lynen Fellowship and by a GIF grant.}
\author{Martino Lupini}
\address{School of Mathematics and Statistics, Victoria University of Wellington,
PO Box 600, Wellington 6140, New Zealand}
\email{martino.lupini@vuw.ac.nz}
\urladdr{http://www.lupini.org/}
\thanks{M.L. was partially supported by the
NSF Grant DMS-1600186, by a Research Establishment Grant from Victoria
University of Wellington, and by a Marsden Fund Fast-Start Grant from the
Royal Society of New Zealand.}
\keywords{Dilation theory, Stinespring's theorem, finite-dimensional space, matrix convexity, Carath\'{e}odory's theorem,
Minkowski's theorem}

\subjclass[2010]{Primary: 47A20; Secondary 46A55, 47L07}

\begin{document}

\begin{abstract}
  We establish a finite-dimensional version of the Arveson--Stinespring dilation theorem
  for unital completely positive maps on operator systems. This result can be seen as a general principle
  to deduce finite-dimensional dilation theorems from their classical infinite-dimensional counterparts.
  In addition to providing unified proofs of known finite-dimensional dilation theorems,
  we establish finite-dimensional versions of Agler's theorem
  on rational dilation on an annulus, of Berger's dilation theorem for operators of numerical
  radius at most $1$, and of the Putinar--Sandberg numerical range dilation theorem.
  As a key tool, we prove versions of Carath\'{e}odory's and of Minkowski's theorem for matrix convex sets.
\end{abstract}

\maketitle

\section{Introduction}

\subsection{Background}
One of the cornerstones of the theory of operators on Hilbert space is Sz.-Nagy's dilation theorem \cite{Sz.-Nagy53},
which can be phrased as follows.

\begin{thm}[Sz.-Nagy]
  \label{thm:sz-nagy}
  Let $T$ be a contraction on a Hilbert space $H$, i.e.\ a linear operator with $\|T\| \le 1$. Then there
  exist a Hilbert space $K \supset H$ and a unitary operator $U$ on $K$ such that, for every polynomial $p$ with complex coefficients,
  \begin{equation}
    \label{eqn:sz-nagy}
    p(T) = P_{H} p(U) \big|_H .
  \end{equation}
\end{thm}
The operator $U$ in Sz.-Nagy's theorem is called a \emph{dilation} of $T$.
This theorem frequently makes it possible to study contractions
through their unitary dilations, the key advantage being that unitaries are well understood
by virtue of the spectral theorem \cite{SFB+10}.
On the other hand, even if $H$ is finite-dimensional, in which case $T$ can be regarded
as a matrix, then the unitary dilation $U$ still typically acts on an infinite-dimensional space $K$.
Indeed, one can show that unless $T$ is itself unitary, $K$ is necessarily infinite-dimensional.
Thus, for contractive matrices $T$, it is not clear that the dilation $U$ is always easier to understand.

This drawback was addressed by Egerv\'ary \cite{Egervary54}, who established a finite-dimensional version of Sz.-Nagy's dilation theorem.

\begin{thm}[Egerv\'ary]
  \label{thm:eger}
  Let $T$ be a contraction on a finite-dimensional Hilbert space $H$ and let $N \in \bN$. Then
  there exist a finite-dimensional Hilbert space $K \supset H$ and a unitary operator $U$ on $K$ such that, for every polynomial $p$ with complex coefficients of degree at most $N$,
  \begin{equation*}
    p(T) = P_H p(U) \big|_H .
  \end{equation*}

\end{thm}

In other words, by only requiring \eqref{eqn:sz-nagy} to hold for a finite-dimensional space
of polynomials, we can retain finite-dimensionality of the dilation space.

Egerv\'ary's theorem was extended to pairs of commuting contractive matrices by \mcc\ and Shalit \cite{MS13};
their result is therefore a finite-dimensional version of And\^o's dilation theorem \cite{Ando63}.
More generally, \mcc\ and Shalit proved a finite-dimensional
dilation theorem for tuples of commuting matrices that admit a dilation to commuting unitaries. This
last result was further generalized by Cohen \cite{Cohen15} to $d$-tuples of commuting operators admitting
a polynomial normal $\partial X$-dilation for compact subsets $X$ of $\bC^d$.
A related finite-dimensional dilation result was proved by Davidson, Dor-On, Shalit and Solel \cite[Theorem 7.1]{DDS+16}.
It is worth remarking that while Egerv\'ary's proof explicitly constructs a unitary
matrix on a larger space, the results of \mcc--Shalit, Cohen and Davidson--Dor-On--Shalit--Solel
all deduce the finite-dimensional dilation theorem from its infinite-dimensional counterpart.
We also refer the reader to the survey article
\cite{LS14}; see also \cite{LM18} for connections of finite-dimensional dilations with quantum
information theory.

\subsection{An abstract finite-dimensional dilation theorem}
The goal of this article is to establish an abstract result that makes it possible
to deduce finite-dimensional dilation theorems from their infinite-dimensional relatives
under general assumptions. In particular, our result will imply all finite-dimensional dilation
theorems mentioned in the preceding paragraph, as well as new ones.

To formulate such an abstract result, the framework of dilations of unital completely positive (u.c.p.) maps
is very useful. Let $A$ be a unital $C^*$-algebra. Recall that an \emph{operator system} is a unital self-adjoint
subspace $S \subset A$. A linear map $\varphi: S \to B(H)$ is said to be positive if it maps
positive elements to positive elements,
and completely positive if all amplifications $\varphi^{(n)}: M_n(S)
\to M_n(B(H))$, defined by applying $\varphi$ entrywise, are positive.
Arveson's extension theorem shows that every u.c.p.\ map $\varphi: S \to B(H)$ extends to a u.c.p.\ map
$\psi: A \to B(H)$. By Stinespring's dilation theorem, $\psi$ dilates to a representation of $A$,
that is, there exist a Hilbert space $K \supset H$ and a unital $*$-homomorphism $\pi: A \to B(K)$
such that $\psi(a) = P_H \pi(a) \big|_H$ for all $a \in A$. In particular,
\begin{equation*}
  \varphi(s) = P_H \pi(s) \big|_H \quad (s \in S).
\end{equation*}
Conversely, every linear map $\varphi: S \to B(H)$ of this form is unital and completely positive.
Seeking finite-dimensional dilations in this setting means asking whether we can achieve that $\dim(K) < \infty$.

In the sequel, we will say that a u.c.p.\ map $\varphi: S \to B(H)$ \emph{dilates to a finite-dimensional
representation of $A$} if there exist a finite-dimensional Hilbert space $K$ containing $H$ and a unital $*$-homomorphism $\pi: A \to B(K)$ such that $\varphi(s) = P_H \pi(s) \big|_H$
for all $s \in S$.

\begin{quest}
  \label{quest:main}
  Let $A$ be a unital $C^*$-algebra, let $S \subset A$ be an operator system and let $\varphi: S \to B(H)$
  be a u.c.p.\ map with $\dim(H) < \infty$. Does $\varphi$ dilate to a finite-dimensional
  representation of $A$?
\end{quest}

As explained earlier, a dilation always exists on a possibly infinite-dimensional space
by Arveson's extension theorem and Stinespring's dilation theorem.

%

Before stating our main result regarding Question \ref{quest:main}, let us observe that the question
can only have a positive answer if the $C^*$-algebra $A$ has ``enough'' finite-dimensional representations.
More precisely, a result of Courtney and Shulman \cite{CS17} implies the following
necessary condition.

\begin{prop}
  \label{prop:nec}
  Let $A$ be a unital $C^*$-algebra with the property that for every
  operator system $S \subset A$ with $\dim(S) \le 2$, each u.c.p.\ map $\varphi: S \to \bC$ dilates
  to a finite-dimensional representation of $A$. Then every irreducible representation
  of $A$ is finite-dimensional.
\end{prop}

\begin{proof}
  Let $a \in A$ and consider the operator system $S = \spa \{1, a^* a\} \subset A$.
By \cite[II.6.3.3]{Blackadar06}, there exists a state $\varphi:S \to \bC$ with $\varphi(a^* a) =\|a^* a\|$.
By assumption, $\varphi$ dilates to a finite-dimensional representation $\pi$ of $A$. Then
\begin{equation*}
  ||\pi(a)||^2 = \|\pi(a^* a)\| \ge \varphi(a^* a) = \|a^* a \|
  = \|a\|^2.
\end{equation*}
Since $\pi$ is contractive, equality holds throughout.
This shows that every element of $A$ attains its norm on a finite-dimensional representation. By a result
of Courtney and Shulman \cite[Theorem 4.4]{CS17}, this is equivalent to saying that every irreducible
representation of $A$ is finite-dimensional.
\end{proof}

$C^*$-algebras whose irreducible representations are all finite-dimensional are called FDI in \cite{CS17}.
We are exclusively concerned with unital $C^*$-algebras, in which
case the class of FDI $C^*$-algebras coincides with the class of liminal (also called CCR) $C^*$-algebras;
see \cite[Section IV.1.3]{Blackadar06}.
Examples of FDI $C^*$-algebras are commutative
$C^*$-algebras, as every irreducible representation of a commutative $C^*$-algebra is one dimensional.
A more general class of examples is given by $r$-subhomogeneous $C^*$-algebras.
These are $C^*$-algebras whose irreducible representations
all occur on a Hilbert space of dimension at most $r$; see \cite[Section IV.1.4]{Blackadar06}.
In \cite{CS17}, examples
of non-subhomogeneous FDI $C^*$-algebras are mentioned, such as full group $C^*$-algebras
of certain Lie groups and algebras arising from mapping telescopes; these can be unitized
if necessary without changing subhomogeneity or the FDI property. Clearly, every FDI $C^*$-algebra
is residually finite-dimensional (RFD), meaning that finite-dimensional representations
separate the elements of the $C^*$-algebra, but the converse is not true.
An example of an RFD $C^*$-algebra that is not FDI is the full group $C^*$-algebra
$C^*(\mathbb{F}_2)$ of the free group on two generators \cite{Choi80}.
For more discussion about FDI $C^*$-algebras, the reader is referred to \cite{CS17}.

Our main result shows that if $A$ is FDI and $\dim(S) < \infty$, then Question \ref{quest:main} has a positive
answer. This result can be regarded as a finite-dimensional version
of the Arveson--Stinespring dilation theorem.

\begin{thm}
  \label{thm:main}
  Let $A$ be a unital FDI (equivalently, unital liminal) $C^*$-algebra, let $S \subset A$ be an operator system with $\dim(S) < \infty$
  and let $\varphi: S \to B(H)$
  be a u.c.p.\ map with $\dim(H) < \infty$.
  Then $\varphi$ dilates to a finite-dimensional representation of $A$.
\end{thm}

This result will be proved as Theorem \ref{thm:FDI} below.
If $A$ is commutative, or more generally subhomogeneous, then we obtain an explicit upper bound for the dimension of the dilation,
see Proposition \ref{prop:quantitative}.
An approximate version of Theorem \ref{thm:main}, in which the $C^*$-algebra $A$ is allowed to be RFD,
is due to Alekseev, Netzer and Thom \cite[Theorem 3.8]{ANT19}.

To illustrate how Theorem \ref{thm:main} can be used to deduce concrete finite-dimensional
dilation theorems from their infinite-dimensional relatives,
let us explain how to prove Egerv\'ary's theorem from Sz.-Nagy's theorem and Theorem \ref{thm:main}

\begin{proof}[Proof of Theorem \ref{thm:eger} from Theorems \ref{thm:sz-nagy} and \ref{thm:main}]
Let $T \in B(H)$ be a contraction with $\dim(H) < \infty$ and let $N \in \bN$.
By Sz.-Nagy's dilation theorem (Theorem \ref{thm:sz-nagy}),
$T$ admits a unitary dilation $V$ on a (generally infinite-dimensional) Hilbert space $L \supset H$.
The continuous functional calculus for $V$ shows that $V$ induces a representation $\sigma: C(\bT) \to B(L)$
with $\sigma(p) = p(V)$ for all $p \in \bC[z]$. Let
\begin{equation*}
  S = \spa \{1, z^k, \ol{z}^k : 1 \le k \le N \} \subset C(\bT),
\end{equation*}
which is a finite-dimensional operator system. Then the map $\varphi: S \to B(H)$ defined by
\begin{equation*}
  \varphi(f) = P_H \sigma(f) \big|_H
\end{equation*}
is u.c.p.\ and satisfies $\varphi(p) = P_H p(V) \big|_H = p(T)$ for all $p \in \bC[z]$ with $\deg(p) \le N$.
Applying Theorem \ref{thm:main} to the commutative $C^*$-algebra $C(\bT)$, we find a Hilbert space
$K \supset H$ with $\dim(K) < \infty$ and a $*$-representation $\pi: C(\bT) \to B(K)$ with
\begin{equation*}
  \varphi(f) = P_H \pi(f) \big|_H \quad (f \in S).
\end{equation*}
Let $U = \pi(z)$. Then $U \in B(K)$ is unitary and
\begin{equation*}
  p(T) = \varphi(p) = P_H p(U) \big|_H
\end{equation*}
for all $p \in \bC[z]$ with $\deg(p) \le N$.
\end{proof}

The above proof shows that, roughly speaking, the operator system $S$ encodes which relations
should hold for the dilation. In particular, the necessity
of the degree bound in Egerv\'ary's theorem shows
that the assumption of finite-dimensionality
of $S$ in Theorem \ref{thm:main} is necessary.

Further applications of Theorem \ref{thm:FDI} will be given in Section \ref{sec:applications}.
In particular, we establish a finite-dimensional dilation theorem for operators
with numerical radius at most $1$ and a finite-dimensional version of Agler's theorem
of rational dilation on an annulus.

Here, we highlight one application regarding matrices with prescribed numerical range. Recall that
the numerical range of an operator $T \in B(H)$ is defined to be
\begin{equation*}
  W(T) = \{ \langle T \xi, \xi \rangle: \xi \in H, \|\xi\| = 1\}.
\end{equation*}
The Toeplitz--Hausdorff theorem shows that $W(T)$ is a convex set.
Moreover, $\sigma(T) \subset \overline{W(T)}$ and $W(T)$ is compact if $H$ is finite-dimensional.
Currently, there is a large amount of activity surrounding the numerical range in the context
of Crouzeix's conjecture \cite{Crouzeix07}, which asserts that
\begin{equation*}
  \|p(T)\| \le 2 \sup_{z \in W(T)} |p(z)|
\end{equation*}
should hold for all polynomials $p$ and all $T \in B(H)$.
Clearly, one may restrict to finite-dimensional Hilbert spaces $H$ here.
For recent work on this problem, see for instance \cite{BGG+20,CP17,RS18} and the references therein.
This conjecture still seems to be open, but it is known that it holds
when the constant $2$ is replaced with $1 + \sqrt{2}$, a result due to Crouzeix and Palencia \cite{CP17}, see also \cite{RS18}. A theorem of Okubo and Ando \cite{OA75} implies that Crouzeix's conjecture
holds with constant $2$ in the case when $W(T)$ is a disc. This result was proved using dilation theory
and hence operator theory in infinite dimensions.
On the other hand, some of the recent progress on Crouzeix's conjecture was obtained
using special properties in finite dimensions, such as the existence
of vectors on which the operator norm is attained; see
for instance \cite{BGG+20,CGL17}.

In this context, we establish the following finite-dimensional dilation theorem, whose infinite-dimensional
counterpart is due to Putinar and Sandberg \cite{PS05}. If $\Omega \subset \mathbb{C}$
is a bounded open set with smooth boundary $\partial \Omega$, let $A(\Omega)$ be the algebra of all holomorphic functions on $\Omega$
that extend to be continuous on $\overline{\Omega}$. If $f \in A(\Omega)$, we let
$C \overline{f}$ be the Cauchy transform of $\overline{f}$, which is defined by
\begin{equation*}
  (C \overline{f})(z) = \frac{1}{2 \pi i} \int_{\partial \Omega} \frac{\overline{f(\zeta)}}{\zeta - z} \, d \zeta
  \quad (z \in \Omega).
\end{equation*}
In particular, $C \overline{f}$ is holomorphic on $\Omega$, so $(C \overline{f})(T)$
is defined whenever $T \in B(H)$ satisfies $\sigma(T) \subset \Omega$.

\begin{thm}
  \label{thm:PS_intro}
  Let $\Omega \subset \mathbb{C}$ be a bounded open convex set with smooth boundary $\partial \Omega$.
  Let $T$ be an operator on a finite-dimensional Hilbert space $H$ with $W(T) \subset \Omega$
  and let $\mathcal{A} \subset A(\Omega)$ be a finite-dimensional subspace.
  Then there exist a finite-dimensional Hilbert space $K \supset H$ and a normal operator $N$
  on $K$ with $\sigma(N) \subset \partial \Omega$ such that
  \begin{equation*}
    f(T) + (C\overline{f})(T)^* = 2 P_H f(N) \big|_H
  \end{equation*}
  for all $f \in \mathcal{A}$.
\end{thm}
This result will be proved in Corollary \ref{cor:PS}.
There, we will also remark on the connection between the dilation result and some of the current approaches
to Crouzeix's conjecture.

\subsection{Matrix convex sets}

To establish our main result, we will use tools from the theory of matrix convexity.
Matrix convex sets were introduced by Wittstock \cite{Wittstock1984} and by Effros and Winkler \cite{Effros1997}
and further studied by Webster and Winkler \cite{WW99}. It is known that dilation
theory is closely related with matrix convexity, see \cite{DDS+16,DK15,FHL16} for some recent work.
In addition, matrix convexity has found applications in real algebraic geometry, see for instance
\cite{HKM16,HM12,Kriel18}.

We will state the precise definition of matrix convex sets in Section \ref{sec:matrix_convex}.
For now, let us simply recall that a matrix convex set $\mathbf{X}$ in a complex vector space $V$
is of the form $\mathbf{X} = (X_n)_{n=1}^\infty$, where $X_n \subset M_n(V)$ for all $n \ge 1$.
There are notions of matrix convex combinations,
matrix convex hull and of matrix extreme points. Moreover, Webster and Winkler \cite{WW99}
proved a version of the Krein--Milman theorem in this setting.

In the article \cite{MS13} of \mcc\ and Shalit and in subsequent works \cite{Cohen15,DDS+16} the
authors crucially use a classical theorem of Carath\'{e}odory from convex analysis (see, for instance, \cite[Theorem 16.1.8]{DD10})
to obtain finite-dimensional dilations.

\begin{thm}[Carath\'{e}odory]
  \label{thm:cara_classical}
  Let $X \subset \bR^n$ be a set. If $x \in \bR^n$ belongs to the convex hull of $X$, then $x$
  is a convex combination of at most $n+1$ points in $X$.
\end{thm}

In the context of matrix convex sets,
Davidson, Dor-On, Shalit and Solel proved a version of Carath\'{e}odory's theorem
for matrix ranges of normal tuples \cite[Theorem 2.7]{DDS+16}.
Kriel established Carath\'{e}odory's theorem
for matrix convex sets consisting of tuples of self-adjoint matrices \cite[Lemma 1.14]{Kriel18}.
For our purposes, the following Carath\'{e}odory theorem for general matrix convex sets will be useful.
It will be proved in Theorem \ref{thm:cara}.

\begin{thm}
  \label{thm:cara_intro}
  Let $V$ be a finite-dimensional vector space and let $\mathbf{X} = (X_n)$ with $X_n \subset M_n(V)$ for $n \ge 1$.
  If $x \in M_n(V)$ belongs to the matrix convex hull of $\mathbf{X}$, then it is
  a matrix convex combination of points of $\mathbf{X}$ of length at most $n^2 (2 \dim(V) + 1)$.
\end{thm}

Carath\'{e}odory's theorem is related with another classical result, due to Minkowski, which can
be thought of as a strengthening of the Krein--Milman theorem in finite dimensions; see for example \cite[Theorem 16.4.6]{DD10}.
The difference with the Krein--Milman theorem is that  closure is not required.

\begin{thm}[Minkowski]
  \label{thm:minkowksi_classical}
  Let $K \subset \bR^n$ be a compact convex set. Then $K$ is the convex hull of its extreme points.
\end{thm}

Kriel obtained a version of Minkowski's theorem in his setting of matrix convex sets;
see Theorem 6.8 in \cite{Kriel18}.
In Theorem \ref{thm:minkowski}, we will prove the following version of Minkowki's theorem for general matrix convex sets,
which will be very useful in the proof of Theorem \ref{thm:main}.

\begin{thm}
  \label{thm:minkowski_intro}
  Let $\mathbf{X}$ be a compact matrix convex set in a finite-dimensional locally convex vector space $V$. Then
  $\mathbf{X}$ is the matrix convex hull of its matrix extreme points.
\end{thm}

For free spectrahedra, a particular class of matrix convex sets,
a recent result of Evert and Helton \cite{EH19} yields
a stronger conclusion than Theorem \ref{thm:cara_intro} and Theorem \ref{thm:minkowski_intro} combined.
In the result of Evert and Helton, it suffices to consider a more restrictive notion of extreme points,
and they obtain a better bound on the length of the matrix convex combination.
However, we will apply Theorem \ref{thm:cara_intro} and Theorem \ref{thm:minkowski_intro} to matrix convex sets
that are typically not free spectrahedra. In the somewhat different setting of $C^*$-convexity, Carath\'{e}odory
and Minkowksi theorems were previously established by Farenick \cite{Farenick92} and Morenz \cite{Morenz94}.

In light of the above mentioned  results, it is not surprising that Carath\'{e}odory's and Minkowski's theorem
hold for general matrix convex sets.
Our contribution to matrix convexity in this article is the introduction of a device
that makes it possible
to relate questions about matrix convexity to questions about classical convexity. Thus,
we are able to deduce Theorem \ref{thm:cara_intro} and Theorem \ref{thm:minkowski_intro} from their classical counterparts.
As a by-product, we also obtain another proof of the Krein--Milman theorem for matrix convex sets
due to Webster and Winkler.

\subsection{Outline}

The remainder of this article is organized as follows. In Section \ref{sec:matrix_convex}, we establish
Carath\'{e}odory's and Minkowski's theorem for matrix convex sets, i.e.\ Theorem \ref{thm:cara_intro}
and Theorem \ref{thm:minkowski_intro}. We also show how our methods yield another proof
of the Krein--Milman theorem due to Webster and Winkler.

In Section \ref{sec:fd_dilation}, we establish our main result, Theorem \ref{thm:FDI},
as well as the explicit dimension bound in the case of subhomogeneous $C^*$-algebras.

Section \ref{sec:applications} consists of applications of the main result to various concrete dilation
problems.

\subsection{Acknowledgements} The authors are grateful to John M\textsuperscript{c}Carthy,
to Michael Dritschel and to an anonymous referee
for asking questions that led to Corollaries \ref{cor:rho_dilation} and \ref{cor:rational_general}.
Moreover, the authors thank David Sherman for bringing
\cite{Morenz94} to their attention.
Finally, the authors greatly appreciate the careful reading and helpful comments of an anonymous referee.

\section{Carath\'{e}odory's and Minkowski's theorem for matrix convex sets}
\label{sec:matrix_convex}

\subsection{Matrix convexity}

Let $V$ be a complex vector space and let $\mathbf{X} = (X_n)_{n=1}^\infty$, where
$X_n \subset M_n(V)$ for all $n \ge 1$. The identification $M_n(V) = M_n \otimes V$
makes it possible to multiply an element $x \in M_n(V)$ with a scalar $k \times n$ matrix
on the left or with a scalar $n \times k$ matrix on the right.
A \emph{matrix convex combination} of elements
$x_i \in M_{n_i}$, where $1 \le i \le s$, is an expression of the form
\begin{equation*}
  x = \sum_{i=1}^s \gamma_i^* x_i \gamma_i,
\end{equation*}
where $\gamma_i \in M_{n_i,n}$ and $\sum_{j=1}^s \gamma_i^* \gamma_i = I_n$.
We refer to the integer $s$ as the \emph{length} of the matrix convex combination.
(Notice that some of the elements $x_i$ may be repeated without reducing the length
of the matrix convex combination.)
The matrix convex combination is called \emph{proper} if
each $\gamma_i$ is surjective, and \emph{trivial} if $n_i = n$ for all $i$ and each
$x_i$ is unitarily equivalent to $x$. An element $x \in X_n$ is said to be a
\emph{matrix extreme point} of $\mathbf{X}$ if whenever $x$ is expressed as
a proper matrix convex combination of elements of $\mathbf{X}$, the matrix convex combination is trivial.
The matrix convex
hull of $\mathbf{X}$ is the smallest matrix convex set that contains $\mathbf{X}$, or equivalently,
the set of all matrix convex combinations of elements of $\mathbf{X}$.
If $\mathbf{X} = (X_n)_{n=1}^\infty$ is a matrix convex set in a topological vector space $V$,
then we endow $M_n(V)$ with the product topology and say that $\mathbf{X}$ is compact (respectively closed) if each $X_n$ is compact (respectively closed).
For more background on matrix convexity and matrix extreme points, see \cite{WW99}.

A \emph{real structure} on $V$ is a conjugate linear involution $*$ on $V$.
If we set
$V_{\bR} = \{v \in V: v = v^*\}$, then $V_{\bR}$ is a real vector space and $V = V_{\bR} + i V_{\bR}$,
hence $\dim_{\bR} V_{\bR} = \dim_{\bC} V$.
A real structure on $V$ induces a real
structure on $M_n(V)$ for all $n \in \bN$, via $[v_{i j}]^* = [ v_{j i}^*]$.
An element $x \in M_n(V)$ is \emph{self-adjoint} if $x = x^*$, and we write $M_n(V)_{sa}$
for the real vector space of all self-adjoint elements of $M_n(V)$.

\begin{exa}
    Let $V = \bC^d$ and consider the involution given by coordinate-wise complex conjugation.
      Then $M_n(V)_{sa}$ can be naturally identified with the set of $d$-tuples of self-adjoint $n \times n$
      matrices. This setting is frequently studied in free convexity; see for example
      \cite{EH19,HKM16,Kriel18}.
\end{exa}

\subsection{Carath\'{e}odory's theorem}

Our goal is to prove versions of Carath\'{e}odory's and Minkowski's theorems
for matrix convex sets, that is, Theorem \ref{thm:cara_intro} and Theorem \ref{thm:minkowski_intro}.
To this end, we will reduce the matrix convex setting to the classical
setting with the help of the following device.
We let $\tr$ denote the normalized trace on $M_n$, so that $\tr(I_n) = 1$.
For $n \ge 1$, we define a subset of $M_n \oplus M_n(V)$ by
\begin{equation*}
  \Gamma_n(\mathbf{X}) = \{ ( \gamma^* \gamma, \gamma^* x \gamma): \gamma \in M_{k,n}, \tr(\gamma^* \gamma) = 1, k\in \bN, x \in X_k\}.
\end{equation*}
This definition should be compared with the definition of $\Delta_n$ in \cite{WW99}, and with
a device in the proof of Theorem 4.7 in \cite{Cohen15}. A similar definition
also occurs in the proof of Proposition 5.5 in \cite{Morenz94} in the context of $C^*$-convexity.

The following simple lemma relates the matrix convex hull of $\mathbf{X}$ to the convex hull
of $\Gamma_n(\mathbf{X})$.
\begin{lem}
  \label{lem:Gamma}
  Let $\mathbf{X} = (X_n)$ with $X_n \subset M_n(V)$ for all $n \ge 1$. Let $x \in M_n(V)$ and let $r \in \bN$.
  Then $x$ is a matrix convex combination of elements of $\mathbf{X}$ of length $r$ if and only 
  if $(I_n,x)$ is a convex combination of $r$ elements of $\Gamma_n(\mathbf{X})$.
\end{lem}

\begin{proof}
  Let $(I_n,x)$ be a convex combination of $r$ elements of $\Gamma_n(\mathbf{X})$, say
  \begin{equation*}
    (I_n,x) = \sum_{j=1}^r t_j (\gamma_j^* \gamma_j, \gamma_j^* x_j \gamma_j).
  \end{equation*}
  Let $\beta_j = t_j^{1/2} \gamma_j$. Then $\sum_{j=1}^r \beta_j^* \beta_j = I_n$ and
  $x = \sum_{j=1}^r \beta_j^* x_j \beta_j$, so $x$ is a matrix convex combination
  of elements of $\mathbf{X}$ of length $r$.

  Conversely, suppose that $x = \sum_{j=1}^r \beta_j^* x_j \beta_j$ is a matrix convex combination
  of elements of $\mathbf{X}$ of length $r$. We may without loss of generality assume that $\beta_j \neq 0$
  for all $j$, so we may define $t_j = \tr(\beta_j^* \beta_j) > 0$ and $\gamma_j = t_j^{-1/2} \beta_j$.
  Then $\tr(\gamma_j^* \gamma_j) = 1$ for all $j$ and
  \begin{equation*}
    (I_n,x) = \sum_{j=1}^r t_j (\gamma_j^*\gamma_j, \gamma_j^* x_j \gamma_j)
  \end{equation*}
  is a convex combination of $r$ elements of $\Gamma_n(\mathbf{X})$.
\end{proof}

We now obtain a more precise version of Theorem \ref{thm:cara_intro}.
\begin{thm}
  \label{thm:cara}
  Let $V$ be a finite-dimensional vector space and let $\mathbf{X} = (X_n)$ with $X_n \subset M_n(V)$ for $n \ge 1$.
  \begin{enumerate}[label=\normalfont{(\alph*)}]
    \item 
  If $x \in M_n(V)$ belongs to the matrix convex hull of $\mathbf{X}$, then it is
  a matrix convex combination of points of $\mathbf{X}$ of length at most $n^2 (2 \dim(V) + 1)$.
\item Suppose that $V$ has a real structure and that $X_n \subset M_n(V)_{sa}$ for $n \ge 1$.
  If $x \in M_n(V)$ belongs to the matrix convex hull of $\mathbf{X}$, then it is a matrix convex
  combination of points of $\mathbf{X}$ of length at most $n^2(\dim(V) + 1)$.
  \end{enumerate}
\end{thm}

\begin{proof}
  (a)
  Since $x$ belongs to the matrix convex hull of $\mathbf{X}$, Lemma \ref{lem:Gamma} implies
  that $(I_n,x)$ belongs to the convex hull of $\Gamma_n(\mathbf{X})$. By definition,
  $\Gamma_n(\mathbf{X})$ is contained in
  \begin{equation*}
    \{ ( \alpha, v): \alpha \in (M_n)_{sa},  \tr(\alpha) = 1, v \in M_n(V) \},
  \end{equation*}
  which is an affine subspace of real dimension $n^2 -1 + 2 n^2 \dim(V)$.
  The classical Carath\'{e}odory theorem
  shows that $(I_n,x)$ is a convex combination of at most $n^2 (2 \dim(V) + 1)$
  points of $\Gamma_n(\mathbf{X})$. Applying Lemma \ref{lem:Gamma} again,
  we find that $x$ is a matrix convex combination of elements of $\mathbf{X}$ of length
  at most $n^2( 2 \dim(V) + 1)$.

  (b) In  the setting of (b), the set $\Gamma_n(\mathbf{X})$ is contained in
  \begin{equation*}
    \{ ( \alpha, v): \alpha \in (M_n)_{sa},  \tr(\alpha) = 1, v \in M_n(V)_{sa} \},
  \end{equation*}
  which is an affine subspace of real dimension $n^2 -1 + n^2 \dim(V)$, so the bound from the
  classical Carath\'{e}odory theorem is $n^2(\dim(V) + 1)$.
\end{proof}

\begin{rem}
  \begin{enumerate}
    \item
      No serious attempt was made to optimize the bounds in Theorem \ref{thm:cara} and we do not
  know if the bounds are sharp. If $n=1$,
  we recover the bounds in the classical Carath\'{e}odory theorem, which are known to be sharp in that case.

\item Kriel's setting in \cite{Kriel18} corresponds to the self-adjoint case of Theorem \ref{thm:cara};
  in Lemma 1.14 of \cite{Kriel18}, he obtains the slightly larger bound $2 n^2 \dim(V) + 1$ in that case (with a different proof).

\item As mentioned in the introduction, Evert and Helton \cite{EH19} obtain a better bound
  in the special case of (absolute) extreme points of compact free spectrahedra. In particular, they
  obtain a bound of the form $2n (\dim(V)+ 1)$ in their setting.
  \end{enumerate}
\end{rem}

As in classical convex analysis, the matrix convex version of Carath\'{e}odory's theorem has consequences
for compactness of matrix convex hulls. This addresses a question raised in \cite[Remark 3.2]{EHK+18}.

\begin{cor}
  \label{cor:cara_compact}
  Let $V$ be a finite-dimensional locally convex
  vector space and let $\mathbf{X} = (X_n)$ with $X_n \subset M_n(V)$ for $n \ge 1$.
  Suppose that each $X_n$ is compact and that $X_n = \emptyset$ for all but finitely many $n \ge 1$.
  Then the matrix convex hull of $\mathbf{X}$ is compact.
\end{cor}

\begin{proof}
  Let $\mathbf{K} = (K_n)_{n=1}^\infty$ be the matrix convex hull of $\mathbf{X}$.
  Write
  \begin{equation*}
    \{ n \ge 1: X_n \neq \emptyset \} = \{n_1,\ldots,n_k \},
  \end{equation*}
  let $n \ge 1$ and let $r = n^2 (2 \dim(V) + 1)$.
  Theorem \ref{thm:cara} implies that for each $n \ge 1$,
  \begin{equation*}
    K_n =
    \Big\{ \sum_{j=1}^k \sum_{i=1}^{r} \gamma_{i j}^* x_{i j} \gamma_{i j} :
    x_{i j} \in X_{n_j}, \gamma_{i j} \in M_{n_j,n} \text{ with } \sum_{j=1}^k \sum_{i=1}^r \gamma_{i j}^* \gamma_{i j} = I_n  \Big\},
  \end{equation*}
  which is easily seen to be compact.
\end{proof}

The following example shows that the assumption that $X_n = \emptyset$ for all but finitely many $n \ge 1$
in Corollary \ref{cor:cara_compact} cannot simply be omitted.

\begin{exa}
  For $n \ge 1$, let $X_n = \{ (1- 1/n) I_n \}$ and let $\mathbf{Y} = (Y_n)_{n=1}^\infty$ be the matrix
  convex hull of $\mathbf{X} = (X_n)_{n=1}^\infty$. It is not hard to check that $Y_1 = [0,1)$, hence
  $\mathbf{Y}$ is not compact.
\end{exa}

\subsection{Minkowski's theorem}

To prove a version of Minkowski's theorem for matrix convex sets, we need the following lemma.
In particular, part (b) shows that if $\mathbf{X}$ is matrix convex,
then in the definition of $\Gamma_n(\mathbf{X})$,
we may assume that each $\gamma$ is surjective and hence $k \le n$.

\begin{lem}
  \label{lem:matrix_convex_Gamma}
  Let $\mathbf{X} = (X_n)$ be a matrix convex set in $V$.
  \begin{enumerate}[label=\normalfont{(\alph*)}]
    \item The set $\Gamma_n(\mathbf{X})$ is convex for all $n \in \bN$.
    \item The set $\Gamma_n(\mathbf{X})$ equals
      \begin{equation*}
          \{ ( \gamma^* \gamma, \gamma^* x \gamma): \gamma \in M_{k,n} \text{ is surjective} , \tr(\gamma^* \gamma) = 1,x \in X_k, k \le n \}.
      \end{equation*}
    \item If $V$ is a topological vector space and if $\mathbf{X}$ is a compact matrix convex set,
      then $\Gamma_n(\mathbf{X})$ is a compact convex set.
  \end{enumerate}
\end{lem}

\begin{proof}
  The arguments are similar to the corresponding arguments in \cite{WW99}.

  (a)
  Let $0 < t < 1$ and let $\gamma_i \in M_{k_i,n}$ and $x_i \in X_{k_i}$ for $i=1,2$ be as in the definition
  of $\Gamma_n(\mathbf{X})$.
  Let $k = k_1 + k_2$ and
  \begin{equation*}
    \gamma =
    \begin{bmatrix}
      t^{1/2} \gamma_1 \\ (1-t)^{1/2} \gamma_2
    \end{bmatrix} \in M_{k,n}.
  \end{equation*}
  Then $\gamma^* \gamma = t \gamma_1^* \gamma_1 + (1 - t) \gamma_2^* \gamma_2$.
  In particular, $\tr(\gamma^* \gamma) = 1$.
  Since $\mathbf{X}$ is matrix convex, $x = x_1 \oplus x_2 \in X_{k}$, so
  \begin{equation*}
    t (\gamma_1^* \gamma_1, \gamma_1^* x_1 \gamma_1) + (1-t) (\gamma_2^* \gamma_2, \gamma_2^* x_2 \gamma_2)
    = (\gamma^* \gamma, \gamma^* x \gamma) \in \Gamma_n(\mathbf{X}).
  \end{equation*}

  (b) Let $\gamma \in M_{k,n}$ and $x \in X_k$ be as in the definition
  of $\Gamma_n(\mathbf{X})$. Let $r$ be the rank of $\gamma$, so that $1 \le r \le n$, and let $\delta \in M_{k,r}$
  be an isometry onto the range of $\gamma$. Define $\beta = \delta^* \gamma \in M_{r,n}$.
  Then $\beta$ is surjective, and
  \begin{equation*}
    (\gamma^* \gamma, \gamma^* x \gamma) =
    (\gamma^* \delta \delta^* \gamma,
    \gamma^* \delta \delta^* x \delta \delta^* \gamma)
    = (\beta^* \beta, \beta^* (\delta^* x \delta) \beta).
  \end{equation*}
  Since $\mathbf{X}$ is matrix convex, $\delta^* x \delta \in X_r$, so we have
  obtained the desired representation.

  (c) We have seen in part (a) that $\Gamma_n(\mathbf{X})$ is convex.
  Part (b) implies that
  \begin{equation*}
    \Gamma_n(\mathbf{X}) = \{ ( \gamma^* \gamma, \gamma^* x \gamma): \gamma \in M_{k,n}, \tr(\gamma^* \gamma) = 1,x \in X_k, k \le n \},
  \end{equation*}
  which shows that $\Gamma_n(\mathbf{X})$ is compact since for each $k$, the set of all $\gamma \in M_{k,n}$
  with $\tr(\gamma^* \gamma) = 1$ is compact.
\end{proof}

The following lemma shows that extreme points of $\Gamma_n(\mathbf{X})$ give
rise to matrix extreme points of $\mathbf{X}$. In fact, we will
see in Proposition \ref{prop:extreme_point_char} that every matrix extreme point arises
in this way, but for the proof of Minkowki's theorem, the easier direction suffices.

\begin{lem}
  \label{lem:extreme_point}
  Let $\mathbf{X} = (X_n)_{n=1}^\infty$ be a matrix convex set in a vector space $V$.
  Let $x \in X_k$ and let $\gamma \in M_{k,n}$ be surjective with $\tr(\gamma^* \gamma) = 1$.
  If $(\gamma^* \gamma, \gamma^* x \gamma)$
  is an extreme point of $\Gamma_n(\mathbf{X})$, then $x$ is a matrix extreme point
  of $\mathbf{X}$.
\end{lem}

\begin{proof}
  Let $x = \sum_{j=1}^s \gamma_j^* x_j \gamma_j$
  be a proper matrix convex combination of $x$ with $\gamma_j \in M_{k_j,k}$ and $x_j \in X_{k_j}$.
  Then
  \begin{equation*}
    (\gamma^* \gamma, \gamma^* x \gamma)
    = \sum_{j=1}^s ( \gamma^* \gamma_j^* \gamma_j \gamma, \gamma^* \gamma_j^* x_j \gamma_j \gamma).
  \end{equation*}
  Since $\gamma$ and $\gamma_j$ are surjective, we may define $t_j = \tr(\gamma^* \gamma_j^* \gamma_j \gamma) > 0$
  and $\beta_j = t_j^{-1/2} \gamma_j \gamma$. Then
  $\tr(\beta_j^* \beta_j) = 1$ for all $j$ and
  \begin{equation*}
    (\gamma^* \gamma, \gamma^* x \gamma) = \sum_{j=1}^s t_j ( \beta_j^* \beta_j, \beta_j^* x_j \beta_j).
  \end{equation*}
  Moreover, $\sum_{j=1}^s t_j = \tr(\gamma^* \gamma) = 1$. Since $(\gamma^* \gamma, \gamma^* x \gamma)$
  is an extreme point of $\Gamma_n(\mathbf{X})$, it follows that
  \begin{equation*}
    (\gamma^* \gamma, \gamma^* x \gamma) = (\beta_j^* \beta_j, \beta_j^* x_j \beta_j)
  \end{equation*}
  for each $j$. Equality in the first component means that
  \begin{equation*}
    \gamma^* \gamma =t_j^{-1} \gamma^* \gamma_j^* \gamma_j \gamma,
  \end{equation*}
  so surjectivity of $\gamma$ implies that $\gamma_j^* \gamma_j = t_j$ for each $j$.
  Since each $\gamma_j$ is also surjective, we find that $k_j = k$ for each $j$, and that
  $t_j^{-1/2} \gamma_j$ is unitary for each $j$. Equality in the second component
  means that
  \begin{equation*}
    \gamma^* x \gamma = \gamma^* (t_j^{-1/2} \gamma_j)^* x_j (t_j^{-1/2} \gamma_j) \gamma,
  \end{equation*}
  so that $x = (t_j^{-1/2} \gamma_j)^* x_j (t_j^{-1/2} \gamma_j)$ by surjectivity
  of $\gamma$. Thus, the matrix convex combination was trivial, so that $x$ is a matrix
  extreme point of $\mathbf{X}$.
\end{proof}

We now are now ready to prove Theorem \ref{thm:minkowski_intro} from the introduction.

\begin{thm}
  \label{thm:minkowski}
  Let $\mathbf{X}$ be a compact matrix convex set in a finite-dimensional locally convex vector space $V$. Then
  $\mathbf{X}$ is the matrix convex hull of its matrix extreme points.
\end{thm}

\begin{proof}
  Let $n \ge 1$ and let $x \in X_n$. Then $\Gamma_n(\mathbf{X})$ is a compact convex
  set in a finite-dimensional space by part (c) of Lemma \ref{lem:matrix_convex_Gamma}. Observe
  that $(I_n,x) \in \Gamma_n(\mathbf{X})$, hence by Minkowski's theorem, $(I_n,x)$
  is a finite convex combination of extreme points of $\Gamma_n(\mathbf{X})$,
  say
  \begin{equation*}
    (I_n,x) = \sum_{j=1}^r t_j (\gamma_j^* \gamma_j, \gamma_j^* x_j \gamma_j).
  \end{equation*}
  By part (b) of Lemma \ref{lem:matrix_convex_Gamma}, we may assume that each $\gamma_j$
  is surjective, so that $x_j \in X_{k_j}$ for some $k_j \le n$.
  In this setting, Lemma \ref{lem:extreme_point} implies that each $x_j$
  is a matrix extreme point of $\mathbf{X}$. Lemma \ref{lem:Gamma}, applied
  to the collection $x_1,\ldots,x_r$,
  shows that $x$ is a matrix convex combination of matrix extreme points
  of $\mathbf{X}$.
\end{proof}

\begin{rem}
  The proof of Theorem \ref{thm:minkowski} shows that each element of $X_n$
  is in fact a matrix convex combination of matrix extreme points
  in $X_k$ for $k \le n$.
\end{rem}

As a by-product, our methods also yield a proof of the Krein--Milman theorem for matrix convex sets due to
Webster and Winkler \cite{WW99}, which is arguably slightly simpler than the original proof.

\begin{thm}[Webster--Winkler]
  Let $\mathbf{X}$ be a compact matrix convex set in a locally convex vector space $V$. Then $\mathbf{X}$
  is the closed matrix convex hull of its matrix extreme points.
\end{thm}

\begin{proof}
  Let $n \ge 1$ and let $x \in X_n$.  By part (c) of Lemma \ref{lem:matrix_convex_Gamma}, $\Gamma_n(\mathbf{X})$
  is a compact convex set in the locally convex space $M_n \oplus M_n(V)$. Since $(I_n,x) \in \Gamma_n(\mathbf{X})$,
  the classical Krein--Milman theorem shows that $(I_n,x)$
  belongs to the closed convex hull of the set of extreme points of $\Gamma_n(\mathbf{X})$.
  Thus, given a neighborhood $U$ of $x$ in $M_n(V)$ and $0 < \varepsilon < 1$,
  there exist $y \in U$ and $\alpha \in M_n$ with $\|I_n - \alpha\| < \varepsilon$
  such that $(\alpha,y)$ is a
  convex combination of extreme points of $\Gamma_n(\mathbf{X})$, say
  \begin{equation*}
    (\alpha,y) = \sum_{j=1}^r t_j (\gamma_j^* \gamma_j, \gamma_j^* x_j \gamma_j).
  \end{equation*}
  By part (b) of Lemma \ref{lem:matrix_convex_Gamma}, we may again assume that each $\gamma_j$ is surjective,
  so that each $x_j$ is a matrix extreme point of $\mathbf{X}$ by Lemma \ref{lem:extreme_point}.
  Note that $\alpha$ is positive and invertible.
  Let $\beta_j = t_j^{1/2} \gamma_j \alpha^{-1/2}$. Then
  \begin{equation*}
    (I_n,\alpha^{-1/2} y \alpha^{-1/2}) = \sum_{j=1}^r ( \beta_j^* \beta_j, \beta_j^* x_j \beta_j),
  \end{equation*}
  hence $\alpha^{-1/2} y \alpha^{-1/2}$ belongs to the matrix convex hull of the matrix extreme points
  $x_1,\ldots,x_r$.
  This is true for every $0 < \varepsilon < 1$, so we can find a sequence $(\alpha_k)$ of positive invertible
  matrices tending to $I_n$ so that
  $\alpha_k^{-1/2} y \alpha_k^{-1/2}$ belongs to the matrix convex hull
  of the matrix extreme points of $\mathbf{X}$ for all $k$. Continuity of the continuous functional calculus (see,
  for instance, \cite[II.2.3.2]{Blackadar06}) shows that $(\alpha_k^{-1/2})$ tends to $I_n$,
  hence $\alpha_k^{-1/2} y \alpha_k^{-1/2}$ tends to $y$ in $M_n(V)$.
  Thus, $y \in U$ belongs to the closure
  of the matrix convex hull of the matrix extreme points of $\mathbf{X}$.
  This is true for every neighborhood $U$ of $x$ in $M_n(V)$,
  from which the result follows.
\end{proof}

\subsection{Matrix extreme points of \texorpdfstring{$\mathbf{X}$}{X} vs.\ extreme points of \texorpdfstring{$\Gamma_n(\mathbf{X})$}{Gamma(X)}}
We will finish this section by establishing the converse of Lemma \ref{lem:extreme_point}, thus showing
that matrix extreme points of $\mathbf{X}$ are in one-to-one correspondence
with extreme points of $\Gamma_n(\mathbf{X})$.
The first step is the following
special case of Arveson's boundary theorem, see for instance \cite[p.\ 889]{Farenick00}.

\begin{lem}
  \label{lem:bdry}
  Let $\alpha_1,\ldots,\alpha_r \in M_n$ with $\sum_{i=1}^r \alpha_i^* \alpha_i = I_n$. If
  \begin{equation*}
    S = \Big\{ \alpha \in M_n : \sum_{i=1}^r \alpha_i^* \alpha \alpha_i = \alpha \Big\}
  \end{equation*}
  is an irreducible set of matrices, then each $\alpha_i$ is a scalar multiple of $I_n$.
\end{lem}

\begin{proof}
  Consider the u.c.p.\ map
  \begin{equation*}
    \varphi: M_n \to M_n, \quad \alpha \mapsto \sum_{i=1}^r \alpha_i^* \alpha \alpha_i.
  \end{equation*}
  Arveson's boundary theorem \cite[Theorem 2.1.1]{Arveson72} implies that the identity representation on $M_n$
  is a boundary representation for $S$,
  meaning in particular that the identity map
  on $S$ admits a unique extension to a u.c.p.\ map from $M_n$ to $M_n$.
  Thus, $\varphi$ is the identity map on $M_n$.
  The uniqueness part in Choi's theorem \cite[Remark 4]{Choi75} then shows that each $\alpha_i$
  is a scalar multiple of $I_n$.
\end{proof}

The following lemma contains a different characterization of matrix extreme points.
It implicitly appears (in a slightly different
setting) in \cite{Farenick00}.

\begin{lem}
  \label{lem:extreme_char_bdry}
  Let $\mathbf{X} = (X_n)$ be a matrix convex set in a vector space $V$ and let $x \in X_n$. The following assertions are equivalent:
  \begin{enumerate}[label=\normalfont{(\roman*)}]
    \item $x$ is a matrix extreme point of $\mathbf{X}$.
    \item Whenever $x = \sum_{i=1}^r \gamma_i^* x_i \gamma_i$ is a matrix convex
      combination of elements of $\mathbf{X}$, then there exist $t_i \ge 0$ with
      $\gamma_i^* \gamma_i = t_i I_n$ and $\gamma_i^* x_i \gamma_i= t_i x$
      for $1 \le i \le r$.
  \end{enumerate}
\end{lem}

\begin{proof}
  (ii) $\Rightarrow$ (i) Let $x= \sum_{i=1}^r \gamma_i^* x_i \gamma_i$ be a proper matrix convex
  combination of elements of $\mathbf{X}$. By assumption, there exist $t_i \ge 0$
  with $\gamma_i^* \gamma_i = t_i I_n$ and $\gamma_i^* x_i \gamma_i = t_i x$
  for $1 \le i \le n$. Since each $\gamma_i$ is surjective, $t_i > 0$. Let $u_i = t_i^{-1/2} \gamma_i$.
  Then $u_i$ is unitary and $u_i^* x_i u_i = x$ for $1 \le i \le n$, so the matrix convex combination
  is trivial.

  (i) $\Rightarrow$ (ii) Suppose that $x$ is a matrix extreme point and let $x = \sum_{i=1}^r \gamma_i^* x_i \gamma_i$
  be a matrix convex combination of elements of $\mathbf{X}$.
  We may assume that $\gamma_i \neq 0$ for $1 \le i \le r$.
  Moreover, as in the proof of part (b) of Lemma \ref{lem:matrix_convex_Gamma},
  there exist elements $\widetilde x_i$ of $\mathbf{X}$ and surjective scalar matrices $\widetilde{\gamma_i}$
  of the appropriate size so that $\gamma_i^* \gamma_i = \widetilde{\gamma_i}^* \widetilde{\gamma_i}$
  and $\gamma_i^* x_i \gamma_i = \widetilde{\gamma_i}^* \widetilde x_i \widetilde \gamma_i$
  for $1 \le i \le n$. Thus, we may without loss of generality assume that each $\gamma_i$
  is surjective, so that the matrix convex combination $x = \sum_{i=1}^r \gamma_i^* x_i \gamma_i$
  is proper.
  
  Since $x$ is a matrix extreme point,
  there exist unitaries $u_i \in M_n$
  so that $x_i = u_i^* x u_i$ for $1 \le i \le r$. Let $\alpha_i = u_i \gamma_i$, so that
  \begin{equation}
    \label{eqn:extreme_char}
    x = \sum_{i=1}^r \alpha_i^* x \alpha_i.
  \end{equation}
  We will show
  that $\alpha_i = \lambda_i I_n$ for some $\lambda_i \in \bC$. Assuming this conclusion for the moment,
  it then follows that $\gamma_i^* \gamma_i =  \alpha_i^* \alpha_i = |\lambda_i|^2 I_n$
  and $\gamma_i^* x_i \gamma_i = \gamma_i^* u_i^* x u_i \gamma_i = \alpha_i^* x \alpha_i = |\lambda_i|^2 x$.

  It remains to show that each $\alpha_i$ is a scalar multiple of $I_n$. Since
  $\sum_{i=1}^r \alpha_i^* \alpha_i = I_n$, it suffices by Lemma \ref{lem:bdry} to prove that
  the operator system
  \begin{equation*}
    S = \Big\{ \alpha \in M_n : \sum_{i=1}^r \alpha_i^* \alpha \alpha_i = \alpha \Big\}
  \end{equation*}
  is irreducible. Let $V^*$ denote the algebraic dual space of $V$ and let
  \begin{equation*}
    S_0 =  \{ (\id_{M_n} \otimes v^*) (x) : v^* \in V^* \} \subset M_n.
  \end{equation*}
  From \eqref{eqn:extreme_char}, we deduce that $S_0 \subset S$.
  We finish the proof by showing that $S_0$ is irreducible.
  Assume toward a contradiction that $S_0$ is reducible. Then there exist
  isometries $\beta \in M_{n k}$ and $\delta \in M_{n l}$
  for some $1 \le k,l < n$,
  so that $\beta \beta^* + \delta \delta^* = I_n$
  and
  \begin{equation*}
    \alpha = \beta \beta^* \alpha \beta \beta^* + \delta \delta^* \alpha \delta \delta^* 
  \end{equation*}
  for all $\alpha \in S_0$. Since maps of the form $(\id_{M_n} \otimes v^*)$ separate the points of $M_n(V)$,
  it follows that
  \begin{equation}
    \label{eqn:extreme_char_2}
    x = \beta (\beta^* x \beta) \beta^* + \delta (\delta^* x \delta) \delta^*.
  \end{equation}
  Matrix convexity of $\mathbf{X}$ implies that $\beta^* x \beta \in X_k$
  and $\delta^* x \delta \in X_l$, so \eqref{eqn:extreme_char_2} expresses
  $x$ as a proper non-trivial matrix convex combination of elements of $\mathbf{X}$, contradicting
  the fact that $x$ is a matrix extreme point of $\mathbf{X}$.
\end{proof}

We are now ready to prove the converse of Lemma \ref{lem:extreme_point}.
\begin{prop}
  \label{prop:extreme_point_char}
  Let $\mathbf{X} = (X_n)_{n=1}^\infty$ be a matrix convex set in a vector space $V$.
  Let $x \in X_k$ and let $\gamma \in M_{k,n}$ be surjective with $\tr(\gamma^* \gamma) = 1$. Then $(\gamma^* \gamma, \gamma^* x \gamma)$
  is an extreme point of $\Gamma_n(\mathbf{X})$ if and only if $x$ is a matrix extreme point
  of $\mathbf{X}$.
\end{prop}

\begin{proof}
  The ``only if'' part is Lemma \ref{lem:extreme_point}.
  Conversely, suppose that $x$ is a matrix extreme point of $\mathbf{X}$ and let
  \begin{equation*}
    (\gamma^* \gamma, \gamma^* x \gamma) = \sum_{j=1}^s t_j (\beta_j^* \beta_j, \beta_j^* x_j \beta_j)
  \end{equation*}
  be a proper convex combination with $\beta_j \in M_{k_j,n}$ surjective, $\tr(\beta_j^* \beta_j) = 1$
  and $x_j \in X_{k_j}$ for each $j$ (which we may assume by part (b) of Lemma \ref{lem:matrix_convex_Gamma}).
  Since $\gamma \in M_{k,n}$ is surjective,
  there exists $\delta \in M_{n,k}$ with $\gamma \delta = I_k$, thus
  \begin{equation*}
    (I_k, x) = \sum_{j=1}^s t_j ( ( \beta_j \delta)^* (\beta_j \delta), (\beta_j \delta)^* x_j (\beta_j \delta)).
  \end{equation*}
  Let $\alpha_j = t_j^{1/2} \beta_j \delta \in M_{k_j,k}$, so that
  \begin{equation*}
    \sum_{j=1}^s \alpha_j^* \alpha_j
    = I_k \quad \text{ and } \quad x = \sum_{j=1}^s \alpha_j^* x_j \alpha_j.
  \end{equation*}
  Since $x$ is a matrix extreme point, it follows from Lemma \ref{lem:extreme_char_bdry}
  that there exist scalars $\lambda_j \ge 0$ such that
  \begin{equation}
    \label{eqn:alpha_trivial}
    \alpha_j^* \alpha_j = \lambda_j I_k \quad \text{ and } \quad
    \alpha_j^* x_j \alpha_j = \lambda_j x.
  \end{equation}

  We claim that
  \begin{equation}
    \label{eqn:del_gam}
    \beta_j \delta \gamma = \beta_j
  \end{equation}
  for each $j$.
  Indeed, since $t_j \beta_j^* \beta_j \le \gamma^* \gamma$, we find that
  $\ker(\gamma) \subset \ker(\beta_j)$. On the other hand, $\gamma( I_n - \delta \gamma) = 0$,
  hence also $\beta_j(I_n -\delta \gamma) = 0$, as asserted.

  Using the definition of $\alpha_j$ and Equations \eqref{eqn:alpha_trivial} and \eqref{eqn:del_gam},
  we conclude that $t_j \beta_j^* \beta_j = \lambda_j \gamma^* \gamma$ and $t_j \beta_j^* x_j \beta_j
  = \lambda_j \gamma^* x \gamma$ for each $j$. Taking traces in the first equation,
  we see that $\lambda_j = t_j$, so that our convex combination
  was trivial.
\end{proof}

\section{A finite-dimensional Arveson--Stinespring theorem}
\label{sec:fd_dilation}

The goal of this section is to prove Theorem \ref{thm:main}. We begin with the following
easy consequence of the usual proof of Stinespring's dilation theorem.

\begin{lem}
  \label{lem:stinespring}
  Let $A$ be a unital $C^*$-algebra, let $S \subset A$ be an operator system
  and let $\varphi: S \to B(H)$ be a u.c.p.\ map with $\dim(H) < \infty$. Then the following
  are equivalent:
  \begin{enumerate}[label=\normalfont{(\roman*)}]
    \item 
      The map $\varphi$ dilates to a finite-dimensional representation of $A$.
  \item There exist a finite-dimensional unital $C^*$-algebra $B$, a unital $*$-homomorphism
    $\sigma: A \to B$ and a u.c.p.\ map $\psi: B \to B(H)$ such that $\varphi(s) = (\psi \circ \sigma)(s)$
    for all $s \in S$.
  \end{enumerate}
  Moreover, in the setting of (ii), we can achieve that $\dim(K) \le \dim(B) \dim(H)$.
\end{lem}

\begin{proof}
  (i) $\Rightarrow$ (ii) Let $\pi: A \to B(K)$ be a dilation of $\varphi$ on a finite-dimensional
  Hilbert space $K$. Then we define $B = B(K)$, $\sigma = \pi$ and $\psi(b) = P_H b \big|_H$.

  (ii) $\Rightarrow$ (i) The usual proof of Stinespring's dilation theorem
  (see, for example, \cite[Theorem 4.1]{Paulsen02})
  shows that in the setting
  of (ii), there exist a Hilbert space $K$ with $\dim(K) \le \dim(B) \dim(H)$ and a unital $*$-homomorphism
  $\tau: B \to B(K)$ such that $\psi(b) = P_H \tau(b) \big|_H$ for all $b \in B$.
  Then $\tau \circ \sigma$ is a finite-dimensional representation of $A$ that dilates $\varphi$.
\end{proof}

Let $S$ be an operator system.
We will apply the results of the preceding section to the matrix state space of $S$,
which is $\mathbf{X} = (X_n)_{n=1}^\infty$,
where
\begin{equation*}
  X_n = \{ \varphi : S \to M_n: \varphi \text{ is u.c.p.} \}.
\end{equation*}
Identifying the space of linear maps from $S$ to $M_n$ with $M_n(S^*)$, the matrix state space $\mathbf{X}$
becomes a weak-$*$ compact matrix convex set in $S^*$.
Elements of $\mathbf{X}$ are also called \emph{matrix states of $S$}. A matrix state $\varphi: S \to M_n$
is said to be \emph{pure} if for every completely positive linear map $\psi: S \to M_n$
for which $\varphi - \psi$ is completely positive, there is a $\lambda \in [0,1]$ with $\psi = \lambda \varphi$.
It is a theorem of Farenick \cite[Theorem B]{Farenick00} that a matrix state of $S$ is pure if and only if
it is a matrix extreme point of the matrix state space of $S$. By a theorem of Arveson \cite[Corollary 1.4.3]{Arveson69},
a matrix state of a unital $C^*$-algebra $A$ is pure if and only if it dilates to an irreducible
representation of $A$.

The following lemma connects Question \ref{quest:main} to matrix convexity.

\begin{lem}
  \label{lem:fd_dil_matrix_conv}
  Let $A$ be a unital FDI $C^*$-algebra, let $S \subset A$ be an operator system
  and let $\varphi: S \to B(H)$ be a u.c.p.\ map with $\dim(H) < \infty$. Then the following are equivalent:
  \begin{enumerate}[label=\normalfont{(\roman*)}]
    \item $\varphi$ dilates to a finite-dimensional representation of $A$.
    \item $\varphi$ is a matrix convex combination of restrictions of pure matrix states of $A$ to $S$.
  \end{enumerate}
  Moreover, if $A$ is $r$-subhomogeneous and the matrix convex combination in (ii) has length $s$,
  then $\varphi$ dilates to a representation of $A$ on a Hilbert space of dimension
  at most $s r^2 \dim(H)$.
\end{lem}

\begin{proof}
  (i) $\Rightarrow$ (ii) Suppose that $\varphi$ dilates
  to a finite-dimensional representation $\pi: A \to B(K)$ of $A$. Thus, there exists an isometry
  $\gamma: H \to K$ so that
  \begin{equation*}
    \varphi(s) = \gamma^* \pi(s) \gamma \quad (s \in S).
  \end{equation*}
  Since $\dim(K) < \infty$, the representation $\pi$ is a finite direct sum of irreducible representations
  $\pi_i: A \to B(K_i)$ of $A$ for $1 \le i \le s$. Then we may regard the isometry $\gamma$ as a column
  \begin{equation*}
    \gamma =
    \begin{bmatrix}
      \gamma_1 \\ \vdots \\ \gamma_s
    \end{bmatrix},
  \end{equation*}
  where $\gamma_i \in B(H,K_i)$, so that
  \begin{equation*}
    \varphi(s) = \sum_{i=1}^s \gamma_i^* \pi_i(s) \gamma_i \quad (s \in S).
  \end{equation*}
  Since irreducible representations of $A$ are pure matrix states
  of $A$ (for instance by \cite[Corollary 1.4.3]{Arveson69}), we see that $\varphi$
  is a matrix convex combination of restrictions of pure matrix states of $A$ to $S$.

  (ii) $\Rightarrow$ (i) Suppose that $\varphi$ is a matrix convex
  combination of restrictions of pure matrix states of $A$ to $S$, say
  \begin{equation*}
    \varphi(s) = \sum_{j=1}^s \gamma_j^* \varphi_j(s) \gamma_j \quad (s \in S),
  \end{equation*}
  where each $\varphi_j: A \to M_{k_j}$ is a pure matrix state.
  Then each $\varphi_j$ dilates to
  an irreducible $*$-representation of $\sigma_j: A \to B(K_j)$ by Corollary 1.4.3 of \cite{Arveson69}.
  Thus, there are isometries $v_j: \bC^{k_j} \to K_j$ such that
  \begin{equation*}
    v_j^* \sigma_j(s) v_j = \varphi_j(s)
  \end{equation*}
  for all $s \in S$.
   Let $\sigma = \sigma_1 \oplus \ldots \oplus \sigma_s$ and let $B = B(K_1) \oplus \ldots \oplus B(K_s)$.
   Since $A$ is FDI, $\dim(K_j) < \infty$ for all $j$, so that $\dim(B) < \infty$.
   Moreover, define
  \begin{equation*}
    \psi: B \to B(H), \quad (b_1,\ldots,b_s) \mapsto \sum_{j=1}^s \gamma_j^* v_j^* b_j v_j \gamma_j.
  \end{equation*}
  Then $\psi$ is u.c.p.\ and $\varphi = \psi \circ \sigma$ on $S$. Thus,
  the implication (ii) $\Rightarrow$ (i) of Lemma \ref{lem:stinespring}
  shows that $\varphi$ dilates to a finite-dimensional representation of $A$.

  To prove the additional assertion, note that if $A$ is $r$-subhomogeneous, then we can assume
  that $\dim(K_j) \le r$ for all $j$, so that $\dim(B) \le sr^2$, hence the dimension bound follows
  from the corresponding dimension bound in Lemma \ref{lem:stinespring}.
\end{proof}

We are now ready to establish our main result, Theorem \ref{thm:main}, which we restate for the reader's convenience.

\begin{thm}
  \label{thm:FDI}
  Let $A$ be a unital FDI $C^*$-algebra, let $S \subset A$ be a finite-dimensional operator system
  and let $\varphi: S \to B(H)$ be a u.c.p map with $\dim(H) < \infty$.
  Then $\varphi$ dilates to a finite-dimensional representation of $A$, that is,
  there exist $K \supset H$ with $\dim(K) < \infty$ and a unital $*$-representation
  $\pi: A \to B(K)$ such that $\varphi(s) = P_H \pi(s) \big|_H$ for all $s \in S$.
\end{thm}

\begin{proof}
  We regard $\varphi$ as an element of the matrix state space of $S$.
  Since $\dim(S) < \infty$, Minkowski's theorem
  for matrix convex sets (Theorem \ref{thm:minkowski}) implies that $\varphi$ is a finite
  matrix convex combination of matrix states that are matrix extreme, say
  \begin{equation*}
    \varphi = \sum_{j=1}^s \gamma_j^* \varphi_j \gamma_j,
  \end{equation*}
  where $\varphi_j: S \to M_{k_j}$.
  By Theorem B in \cite{Farenick00}, each $\varphi_j$ extends to a pure matrix state on $A$.
  Thus, the implication (ii) $\Rightarrow$ (i) of Lemma \ref{lem:fd_dil_matrix_conv} shows
  that $\varphi$ dilates to a finite-dimensional representation of $A$.
\end{proof}

We also obtain the following quantitative bound in the preceding result in the case of subhomogeneous $C^*$-algebras.

\begin{prop}
  \label{prop:quantitative}
  In the setting of Theorem \ref{thm:FDI}, suppose that $A$ is also $r$-subhomogeneous.
  Then we may achieve that
  \begin{equation*}
    \dim(K) \le r^2 \dim(H)^3 (\dim(S)+1).
  \end{equation*}
  In particular, if $A$ is commutative, we may achieve that
  \begin{equation*}
    \dim(K) \le \dim(H)^3 (\dim(S)+1).
  \end{equation*}
\end{prop}

\begin{proof}
  We bound the length $s$ of the matrix convex combination in the proof of Theorem \ref{thm:FDI}
  using Carath\'{e}odory's theorem for matrix convex sets. To this end, recall that the matrix state
  space of $S$ is a matrix convex set in $S^*$. Moreover, $S^*$ has a real structure, given
  by the involution
  \begin{equation*}
    \varphi^*(s) = \ol{\varphi(s^*)} \quad (\varphi \in S^*, s \in S).
  \end{equation*}
  The induced involution on $M_n(S^*) = \operatorname{Hom}(S,M_n)$ is given by $\varphi^*(s)
  = \varphi(s^*)^*$, where $\varphi : S \to M_n$ and $s \in S$.
  Therefore, matrix states of $S$ are self-adjoint with respect
  to the real structure, so part (b) of Theorem \ref{thm:cara}
  applies and yields for the length $s$ of the matrix convex combination the bound
  \begin{equation*}
    s \le \dim(H)^2 (\dim(S)+1).
  \end{equation*}
  Thus, the dimension bound follows from the corresponding bound in Lemma \ref{lem:fd_dil_matrix_conv}.
\end{proof}

\section{Applications}
\label{sec:applications}
In this section, we will explore several consequences of Theorem \ref{thm:main} to
concrete dilation problems in operator theory.

\subsection{Known finite-dimensional dilation theorems}

We already
explained in the introduction how to obtain Egerv\'ary's theorem from Theorem \ref{thm:main}.
In fact, the argument proves a more general result.
Let $A(\bD)$ denote the disc algebra, that is, the algebra
of all holomorphic functions on $\bD $ that extend continuously
to $\ol{\bD}$. Sz.-Nagy's dilation theorem (or von Neumann's inequality) shows
that every contraction $T$ has an $A(\bD)$-functional calculus.

\begin{cor}
  Let $T$ be a contraction on a finite-dimensional Hilbert space $H$
  and let $\cA \subset A(\bD)$ be a finite-dimensional subspace. Then there
  exist a finite-dimensional Hilbert space $K \supset H$ and a unitary operator
  $U$ on $K$ such that
  \begin{equation*}
    f(T) = P_H f(U) \big|_H
  \end{equation*}
  for all $f \in \cA$.
  We may achieve that $\dim(K) \le  2 \dim(H)^3 (\dim(\mathcal{A})+1)$.
\end{cor}

\begin{proof}
  We argue exactly as in the introduction, but this time using the operator system
  \begin{equation*}
    S = \spa \{ 1 ,f, \ol{f} : f \in \cA \} \subset C(\bT).
  \end{equation*}
  The dimension bound follows from Proposition \ref{prop:quantitative} as $\dim(S) \le 2 \dim(\mathcal{A})+1$.
\end{proof}


A similar argument proves the theorem of \mcc--Shalit \cite[Theorem 1.2]{MS13}
about dilations of tuples of commuting contractive
matrices.

\begin{cor}[\mcc--Shalit]
  \label{cor:ms_dil}
  Let $T = (T_1,\ldots,T_d)$ be a tuple of
  commuting contractions on a finite-dimensional Hilbert space $H$
  that dilates to a tuple of commuting unitaries.
  Let $\cP \subset \bC[z_1,\ldots,z_d]$ be a finite-dimensional subspace. 
  Then there exist a finite-dimensional Hilbert space $K \supset H$ and a tuple
  of commuting unitaries $U = (U_1,\ldots,U_d)$ on $K$ such that
  \begin{equation*}
    p(T) = P_H p(U) \big|_H
  \end{equation*}
  for all $p \in \cP$.
  We may achieve that $\dim(K) \le  2 \dim(H)^3 (\dim(\mathcal{P})+1)$.
\end{cor}

\begin{proof}
  We apply Theorem \ref{thm:main} to the $C^*$-algebra $A = C(\bT^d)$ and the operator
  system
  \begin{equation*}
    S = \spa \{1, f, \ol{f} : f \in \cP \} \subset C(\bT^d).
  \end{equation*}
  If $V = (V_1,\ldots,V_d)$ is a tuple of commuting unitaries on $L \supset H$ that dilates $T$,
  then $V$ induces a $*$-representation
  \begin{equation*}
    \sigma: C(\bT^d) \to B(L) \quad \text{ with } \sigma(p) = p(V)
  \end{equation*}
  for all $p \in \bC[z_1,\ldots,z_d]$, hence
  \begin{equation*}
    \varphi: S \to B(H), \quad f \mapsto P_H \sigma(f) \big|_H,
  \end{equation*}
  is u.c.p.\ and satisfies $\varphi(p) = p(T)$ for all $p \in \cP$. Theorem \ref{thm:main} yields
  a finite-dimensional Hilbert space $K \supset H$ and a dilation $\pi: C(\bT^d) \to B(K)$ of $\varphi$.
  Defining $U_i = \pi(z_i)$ for $1 \le i \le d$, we obtain the desired dilation.
  The dimension bound once again follows from Proposition \ref{prop:quantitative}.
\end{proof}

As mentioned in \cite{MS13}, the existence of a unitary dilation is automatic if $d = 2$ by And\^o's dilation
theorem.

\mcc\ and Shalit also prove a theorem regarding regular dilations.
This is a stronger notion of dilation to commuting unitaries. While
there is no simple characterization of those tuples of commuting contractions
that admit a unitary dilation, a clean
characterization of those tuples that admit a regular dilation
is known, see \cite[Section I.9]{SFB+10}.

Let $T = (T_1,\ldots,T_d)$ be a tuple of commuting contractions on $H$. If $n \in \bZ^n$, let $n^+ = \max(n,0)$
and $n^- = -\min(n,0)$, where $\max$ and $\min$ are understood entrywise. Thus, $n^+$
is the $d$-tuple of non-negative integers obtained from $n$ by setting all negative entries equal to $0$,
and $n^-$ is the $d$-tuple of non-negative integers obtained from $-n$ by setting all negative entries
equal to $0$. Define $T(n) = (T^*)^{n^-} T^{n^+}$ for $n \in \bZ$. With this definition,
a regular unitary dilation of $T$ is a tuple of commuting unitaries $U$ on a Hilbert space $K \supset H$
such that
\begin{equation*}
  T(n) = P_H U^n \big|_H
\end{equation*}
for all $n \in \bZ^d$.
We can also obtain the result of \mcc\ and Shalit regarding finite
dimensional regular dilations \cite[Theorem 1.7]{MS13} from Theorem \ref{thm:main}.

\begin{cor}[\mcc--Shalit]
  Let $T = (T_1,\ldots,T_d)$ be a tuple of commuting contractions
  on a finite-dimensional Hilbert space $H$ that admits a regular unitary dilation.
  Let $Z \subset \bZ^d$ be a finite subset. Then there exist a finite-dimensional Hilbert space $K \supset H$
  and a tuple of commuting unitaries $U=(U_1,\ldots,U_d)$ on $K$ such that
  \begin{equation*}
    T(n) = P_H U^n \big|_H
  \end{equation*}
  for all $n \in Z$.
  We may achieve that $\dim(K) \le  2 \dim(H)^3 (|Z|+1)$.
\end{cor}

\begin{proof}
  We apply Theorem \ref{thm:main} to the $C^*$-algebra $A = C(\bT^d)$, the operator system
  \begin{equation*}
    S = \spa \{ 1, z^n , \ol{z}^n: n \in Z \} \subset C(\bT^d)
  \end{equation*}
  and the unital map $\varphi: S \to B(H)$ defined by
  $\varphi(z^n) = T(n)$ for $n \in Z \cup -Z \cup \{0\}$, and extended linearly.
  The assumption that $T$ admits a regular unitary dilation shows that $\varphi$ dilates
  to a representation of $C(\bT^d)$, and hence is completely positive.
  Theorem \ref{thm:main} yields a finite-dimensional Hilbert space $K \supset H$ and
  a dilation $\pi: A \to B(K)$ of $\varphi$, so defining $U_i = \pi(z_i)$ for $1 \le i \le d$
  as before and appealing to Proposition \ref{prop:quantitative} for the dimension bound finishes the proof.
\end{proof}

Corollary \ref{cor:ms_dil} was extended by Cohen \cite{Cohen15} in the following way.
Let $X \subset \bC^d$ be a compact set and let $T = (T_1,\ldots,T_d)$ be a tuple of commuting
operators on $H$. A polynomial normal $\partial X$-dilation of $T$ is a $d$-tuple
of commuting normal operators $N = (N_1,\ldots,N_d)$ on a Hilbert space $K \supset H$
with $\sigma(N) \subset \partial X$ such that
\begin{equation*}
  p(T) = P_H p(N) \big|_H
\end{equation*}
for all $p \in \bC[z_1,\ldots,z_d]$.
(Here, the spectrum is computed in the unital commutative $C^*$-algebra generated by $N_1,\ldots,N_d$.)
Our abstract dilation result also implies Cohen's finite-dimensional
dilation theorem.

\begin{cor}[Cohen]
  \label{cor:Cohen}
  Let $T = (T_1,\ldots,T_d)$ be a tuple of commuting operators on a finite-dimensional Hilbert space $H$
  that admits a polynomial normal $\partial X$-dilation. Let $\cP \subset \bC[z_1,\ldots,z_d]$ be a finite
  dimensional subspace. Then there exist a finite-dimensional Hilbert space $K \supset H$ and a tuple
  $N=(N_1,\ldots,N_d)$ of commuting normal operators on $K$ with $\sigma(N) \subset \partial X$ such that
\begin{equation*}
  p(T) = P_H p(N) \big|_H
\end{equation*}
for all $p \in \cP$.
  We may achieve that $\dim(K) \le  2 \dim(H)^3 (\dim(\mathcal{P})+1)$.
\end{cor}

\begin{proof}
  We apply Theorem \ref{thm:main} and Proposition \ref{prop:quantitative} to the $C^*$-algebra $A = C(\partial X)$, the
  operator system
  \begin{equation*}
    S = \spa \{1 , p ,\ol{p} : p \in \cP \} \subset C(\partial X) 
  \end{equation*}
  and the unique u.c.p.\ map $\varphi: S \to B(H)$ satisfying $\varphi(p) = p(T)$
  for all $p \in \cP$.
\end{proof}

We also obtain the following result of Davidson, Dor-On, Shalit and Solel \cite[Theorem 7.1]{DDS+16}
as a consequence.

\begin{cor}[Davidson--Dor-On--Shalit--Solel]
  Let $X = (X_1,\ldots,X_d)$ be a tuple of (not necessarily commuting) operators on a finite-dimensional
  Hilbert space $H$ for which there exist a Hilbert space $L \supset H$ and a tuple $N = (N_1,\ldots,N_d)$
  of commuting normal operators on $L$ such that $X_i = P_H N_i \big|_H$ for $1 \le i \le d$.
  Then there exist a finite-dimensional Hilbert space $K \supset H$ and a tuple $Y = (Y_1,\ldots,Y_d)$
  of commuting normal operators on $K$ with $\sigma(Y) \subset \sigma(N)$ such that
  $X_i = P_H Y_i \big|_H$ for $1 \le i \le d$. We may achieve that $\dim(K) \le 2 \dim(H)^3 (d+1)$.
\end{cor}

\begin{proof}
  Let $X = \sigma(N)$.
  We apply Theorem \ref{thm:main} and Proposition \ref{prop:quantitative} to the $C^*$-algebra $A = C(X)$, the operator system
  \begin{equation*}
    S = \spa \{1, z_i, \ol{z_i}: 1 \le i \le d \} \subset C(X)
  \end{equation*}
  and the u.c.p.\ map $\varphi: S \to B(H)$ defined by $\varphi(1) = 1$, $\varphi(z_i) = X_i$
  and $\varphi(\ol{z_i}) = X_i^*$, extended linearly.
\end{proof}

\subsection{Rational dilation}

To illustrate how Theorem \ref{thm:FDI} can be used to prove new finite-dimensional
dilation results, we establish a finite-dimensional version of Agler's theorem \cite{Agler85a}.
For $0 < r < 1$, let
\begin{equation*}
  A_r = \{ z \in \bC: r \le |z| \le 1 \}
\end{equation*}
and let $\Rat(A_r)$ denote the vector space of all rational functions with poles off $A_r$.
If $T$ is a bounded operator on $H$ with $\sigma(T) \subset A_r$,
we say that $A_r$ is a \emph{spectral set for $T$} if $\|f(T)\| \le
\sup_{z \in A_r} |f(z)|$ for all $f \in \Rat(A_r)$.

\begin{cor}
  \label{cor:agler}
  Let $T$ be an operator on a finite-dimensional Hilbert space $H$ such that $A_r$ is a spectral set for $T$.
  Let $\cR \subset \Rat(A_r)$ be a finite-dimensional subspace.
  Then there exist a finite-dimensional Hilbert space $K \supset H$ and a normal
  operator $N$ on $K$ with $\sigma(N) \subset \partial A_r$ such that
  \begin{equation*}
    f(T) = P_H f(N) \big|_H
  \end{equation*}
  for all $f \in \cR$. We may achieve that $\dim(K) \le 2 \dim(H)^3(\dim(\mathcal{R}) + 1)$.
\end{cor}

\begin{proof}
  We apply Theorem \ref{thm:main} with $A = C(\partial A_r)$ and
  \begin{equation*}
    S = \spa \{ 1, f , \ol{f}: f \in \cR \} \subset C(\partial A_r).
  \end{equation*}
  By Agler's theorem \cite{Agler85a}, there exist a Hilbert space $L \supset H$ and a normal operator $B$
  on $L$ with $\sigma(B) \subset \partial A_r$ so that $f(T) = P_H f(B) \big|_H$ for all $f \in \Rat(A_r)$.
  Since $B$ induces a representation of $C(\partial A_r)$, there exists a u.c.p.\ map
  $\varphi: S \to B(H)$ with $\varphi(f) = f(T)$ for all $f \in \cR$. By Theorem \ref{thm:main},
  $\varphi$ dilates to a finite-dimensional representation $\pi$ of $C(\partial A_r)$,
  so
  \begin{equation*}
    f(T) = \varphi(f) = P_H \pi(f) \big|_H
  \end{equation*}
  for $f \in \cR$.
  If we define
  $N = \pi(z)$, then $f(N) = \pi(f)$ for all $f \in \Rat(A_r)$ since $\pi$ is a homomorphism,
  so $N$ has the required properties. The dimension bound is once again a consequence of Proposition \ref{prop:quantitative}.
\end{proof}

We can in particular apply Corollary \ref{cor:agler} for each $k \in \bN$ to the space
$\cR = \spa \{ z^n: -k \le n \le k \}$
to obtain a normal operator $N$ on a finite-dimensional space with
\begin{equation*}
  T^n = P_H N^n \big|_H
\end{equation*}
for all $-k \le n \le k$.

For compact subsets $X \subset \bC$ with more than one hole, it is in general no longer true
that every operator for which $X$ is a spectral set dilates to a normal operator with spectrum
in $\partial X$; see \cite{AHR08,DM05a}.
In fact, there are typically finite-dimensional counterexamples; see \cite[Section 7]{DM05a}.

For general compact subsets $X$ of $\mathbb{C}$ or of $\mathbb{C}^d$,
Theorem \ref{thm:main} implies a version of Corollary \ref{cor:Cohen}
for rational dilation.
The authors are grateful to Michael Dritschel and to an anonymous referee
for asking questions that led to the inclusion of this result.

Let $X \subset \mathbb{C}^d$ be compact and let
\begin{equation*}
  \Rat(X) = \Big\{ \frac{p}{q} : p, q \in \mathbb{C}[z_1,\ldots,z_d] \text{ and } q(z) \neq 0 \text{ for all } z \in X \Big\}.
\end{equation*}
Let $T = (T_1,\ldots,T_d)$ be a tuple of commuting operators on $H$
whose Taylor spectrum $\sigma_{T}(T)$ is contained in $X$ (see \cite[Chapter IV]{Mueller07} for background
on the Taylor spectrum; in finite dimensions, the Taylor spectrum agrees with various other notions of spectrum).
A rational normal $\partial X$-dilation of $T$ is a $d$-tuple of commuting normal operators
$N = (N_1,\ldots,N_d)$ on a Hilbert space $K \supset H$ with $\sigma(N) \subset \partial X$
such that
\begin{equation*}
  f(T) = P_H f(N) \big|_H
\end{equation*}
for all $f \in \Rat(X)$. (Here, $f(T)$ can be defined by using that $q(T)$ is invertible
if $q$ is a polynomial that does not vanish on $X$, which follows from the spectral
mapping property of the Taylor spectrum; see \cite[Corollary 30.11]{Mueller07}.)

\begin{cor}
  \label{cor:rational_general}
  Let $X \subset \bC^d$ be a compact set and let $T$ be an operator on a finite-dimensional
  Hilbert space $H$
  with $\sigma_T(T) \subset X$ that admits a rational normal $\partial X$-dilation.
  Let $\cR \subset \Rat(X)$ be a finite-dimensional subspace. Then there exist
  a finite-dimensional Hilbert space $K \supset H$ and a tuple $N = (N_1,\ldots,N_d)$
  of commuting normal operators on $K$
  with $\sigma(N) \subset \partial X$ such that
  \begin{equation*}
    f(T) = P_H f(N) \big|_H
  \end{equation*}
  for all $f \in \cR$. We may achieve that $\dim(K) \le 2 \dim(H)^3 (\dim(\mathcal{R}) + 1)$.
\end{cor}

\begin{proof}
  We apply Theorem \ref{thm:main} with $A = C(\partial X)$,
  \begin{equation*}
    S = \spa \{1, f, \overline{f}: f \in \mathcal{R} \} \subset C(\partial X)
  \end{equation*}
  and the unique u.c.p.\ map $\varphi: S \to B(H)$ satisfying $\varphi(f) = f(T)$
  for all $f \in \mathcal{R}$. Thus, we obtain a finite-dimensional
  representation $\pi$ of $C(\partial A_r)$ with
  \begin{equation*}
    f(T) = \varphi(f) = P_H \pi(f) \big|_H
  \end{equation*}
  for all $f \in \mathcal{R}$. Defining $N_i = \pi(z_i)$ for $1 \le i \le d$, we
  obtain a tuple $N$ of commuting normal operators with $\sigma(N) \subset \partial X$.
  Since $\pi$ is a homomorphism, $\pi(f) = f(N)$ for all $f \in \Rat(X)$, so $N$
  has all desired properties. The dimension bound once again follows from Proposition \ref{prop:quantitative}.
\end{proof}

\subsection{Unitary \texorpdfstring{$\rho$}{rho}-dilations}

Let $T \in B(H)$ and $\rho > 0$. A \emph{unitary $\rho$-dilation} of $T$ is a unitary operator $U$ on a Hilbert space
$K \supset H$ such that
\begin{equation*}
  T^n = \rho P_H U^n \big|_H \quad \text{ for all } n \ge 1.
\end{equation*}
The class of operators $C_\rho$ that admit a unitary $\rho$-dilation can be characterized
intrinsically, see \cite[Theorem 11.1]{SFB+10}. In particular, $C_1$ consists of all contractions,
and $C_2$ consists of all operators whose numerical radius is at most $1$.
We can also establish the existence of finite-dimensional $\rho$-dilations.
The authors are grateful to John M\textsuperscript{c}Carthy for asking a question that led
to this observation.

\begin{cor}
  \label{cor:rho_dilation}
  Let $T$ be an operator on a finite-dimensional Hilbert space $H$ and let $\rho > 0$.
  Suppose that $T$ admits a unitary $\rho$-dilation and let $N \in \bN$. Then there exist
  a finite-dimensional Hilbert space $K \supset H$ and a unitary operator $U$ on $K$ such that
  \begin{equation*}
    T^n = \rho P_H U^n \big|_H \quad \text{ for } 1 \le n \le N.
  \end{equation*}
  We may achieve that $\dim(K) \le 2 \dim(H)^3 (N+1)$.
\end{cor}

\begin{proof}
  As in the proof of Eger\'ary's theorem, we apply Theorem \ref{thm:main} to
  $A = C(\bT)$ and the operator system
  \begin{equation*}
    S = \spa \{ 1 , z^n, \ol{z}^n: 1 \le n \le N \} \subset C(\bT),
  \end{equation*}
  but to a different u.c.p.\ map. Let $V$ be a unitary $\rho$-dilation on a Hilbert
  space $K \supset H$, let $\sigma: C(\bT) \to B(L)$ be the corresponding
  representation satisfying $\sigma(p) = p(V)$ for all $p \in \bC[z]$ and let
  \begin{equation*}
    \varphi: S \to B(H), \quad f \mapsto P_H \sigma(f) \big|_H.
  \end{equation*}
  Then $\varphi$ is u.c.p.\ and satisfies $\varphi(z^n) = \rho^{-1} T^n$ for $1 \le n \le N$.
  By Theorem \ref{thm:main}, there exist a finite-dimensional Hilbert space $K \supset H$
  and a representation $\pi: C(\bT) \to B(K)$ that dilates $\varphi$. Let $U = \pi(z)$. Then $U$
  is unitary and $T^n = \rho \varphi(z^n) = \rho P_H U^n \big|_H$ for all $1 \le n \le N$.
  The dimension bound follows from Proposition \ref{prop:quantitative}.
\end{proof}

In particular, setting $\rho = 2$, we obtain the following finite-dimensional version of Berger's dilation
theorem \cite{Berger65}.

\begin{cor}
  \label{cor:berger}
  Let $T$ be an operator on a finite-dimensional Hilbert space $H$ with numerical radius
  at most $1$. Let $N \in \bN$.
   Then there exist
  a finite-dimensional Hilbert space $K \supset H$ and a unitary operator $U$ on $K$ such that
  \begin{equation*}
    \pushQED{\qed}
    T^n = 2 P_H U^n \big|_H \quad \text{ for } 1 \le n \le N. \qedhere
    \popQED
  \end{equation*}
\end{cor}

\subsection{Numerical range dilations}

Next, we establish Theorem \ref{thm:PS_intro} regarding dilations of operators with prescribed numerical range.
As mentioned in the introduction, this is a finite-dimensional version of a theorem of Putinar and Sandberg;
see Theorem 2 and the discussion following it in \cite{PS05}.
It generalizes Corollary \ref{cor:berger}, which corresponds to the case where the set
$\Omega$ below is the unit disc.
For the reader's convenience, we restate the result.

\begin{cor}
  \label{cor:PS}
  Let $\Omega \subset \mathbb{C}$ be a bounded open convex set with smooth boundary $\partial \Omega$.
  Let $T$ be an operator on a finite-dimensional Hilbert space $H$ with $W(T) \subset \Omega$
  and let $\mathcal{A} \subset A(\Omega)$ be a finite-dimensional subspace.
  Then there exist a finite-dimensional Hilbert space $K \supset H$ and a normal operator $N$
  on $K$ with $\sigma(N) \subset \partial \Omega$ such that
  \begin{equation*}
    f(T) + (C\overline{f})(T)^* = 2 P_H f(N) \big|_H
  \end{equation*}
  for all $f \in \mathcal{A}$. We may achieve that $\dim(K) \le 2 \dim(H)^3 (\dim(\mathcal{A}) + 1)$.
\end{cor}

\begin{proof}
  We apply Theorem \ref{thm:main} with $A = C(\partial \Omega)$ and
  \begin{equation*}
    S = \spa \{ 1, f, \overline{f}: f \in \mathcal{A} \}  \subset C(\partial \Omega).
  \end{equation*}
  If $f \in A(\Omega)$, then by the Riesz--Dunford functional calculus,
  \begin{align*}
    f(T) + (C \overline{f})(T)^* &=
    \frac{1}{2 \pi i} \int_{\partial \Omega} f(\zeta) (\zeta - T)^{-1} d \zeta
    + \Big( \frac{1}{2 \pi i} \int_{\partial \Omega} \overline{f(\zeta)}
    (\zeta - T)^{-1} d \zeta \Big)^* \\
    &= 2 \int_{\partial \Omega} f(\zeta) d \mu_T(\zeta),
  \end{align*}
  where $\mu_T$ is the operator-valued measure on $\partial \Omega$ given by
  \begin{equation*}
    d \mu_T(\zeta) = \Re \Big( \frac{1}{2 \pi i} (\zeta - T)^{-1} d \zeta \Big).
  \end{equation*}
  The fact that $W(T) \subset \Omega$ implies that $\mu_T$ is a positive measure;
  see \cite[Section 3]{PS05} or \cite[Section 2]{CP17}. Thus, the map
  \begin{equation*}
    \varphi: S \to B(H), \quad f \mapsto \int_{\partial \Omega} f(\zeta) d \mu_T(\zeta),
  \end{equation*}
  is u.c.p.\ (for instance by \cite[Theorem 3.11]{Paulsen02}) and satisfies $2 \varphi(f) = f(T) + (C \overline{f})(T)^*$ for all
  $f \in \mathcal{A}$. By Theorem \ref{thm:main}, there exists a finite-dimensional
  Hilbert space $K \supset H$ and a representation $\pi: C(\partial \Omega) \to B(K)$
  dilating $\varphi$. Let $N = \pi(z)$.
  Then $N$ is a normal operator with $\sigma(N) \subset \partial \Omega$ and
  \begin{equation*}
    f(T) + (C \overline{f})(T)^*
    = 2 P_H \pi(f) \big|_H = 2 P_H f(N) \big|_H
  \end{equation*}
  for all $f \in \mathcal{A}$. Proposition \ref{prop:quantitative} yields the dimension bound.
\end{proof}

It seems worth remarking that positivity of the operator-valued measure $\mu_T$ in the above
proof plays a crucial role in most of the current approaches to Crouzeix's conjecture; see \cite{RS18}
for a very clear explanation. The arguments of Putinar and Sandberg \cite[Section 3]{PS05}
and the above proof show that this is actually closely related to a dilation result by the Arveson--Stinespring
dilation theorem.

\subsection{\texorpdfstring{$q$}{q}-commuting contractions}

We finish this section with an application in which the $C^*$-algebra $A$ in Theorem \ref{thm:main}
is non-commutative. Let $q$ be a complex number of modulus one.
Two operators $T_1,T_2$ on $H$ are said to be \emph{$q$-commuting} if
\begin{equation*}
  T_2 T_1 = q T_1 T_2.
\end{equation*}
In particular, if $q=-1$, then $T_1$ and $T_2$ anti-commute. It was shown by Keshari and Mallick \cite{KM19},
extending previous work of Sebesty\'en \cite{Sebestyen94}, that any pair of $q$-commuting contractions
dilates to a pair of $q$-commuting unitaries. We can also establish a finite-dimensional
version of their dilation theorem.

\begin{cor}
  \label{cor:q_commuting}
  Let $q = \exp(2 \pi i a/b)$, where $a \in \mathbb{Z}$ and $b \in \mathbb{N} \setminus \{0\}$.
  Let $T_1,T_2$ be $q$-commuting contractions on a finite-dimensional Hilbert space $H$.
  Let $N \in \bN$. Then there exist a finite-dimensional Hilbert space $K \supset H$ and
  $q$-commuting unitaries $U_1,U_2$ on $K$ so that
  \begin{equation*}
    T_1^n T_2^m = P_H U_1^n U_2^m \big|_{H} \quad \text{ for all } 0 \le m,n \le N.
  \end{equation*}
  We may achieve that $\dim(K) \le 2 b^2 \dim({H})^3 (N+1)^2$.
\end{cor}

\begin{proof}
  By \cite[Theorem 2.3]{KM19}, there exist a Hilbert space $L \supset H$ and $q$-commuting unitaries
  $V_1,V_2$ on $L$ so that
  \begin{equation*}
    T_1^n T_2^m = P_H V_1^n V_2^m \big|_H \quad \text{ for all } n,m \in \bN.
  \end{equation*}
  Let $\cA_{a/b}$ be the rational rotation algebra, that is, the universal $C^*$-algebra
  generated by two $q$-commuting unitaries $u_1,u_2$. By \cite[Proposition 1]{DeBrabanter84}, $\cA_{a/b}$
  is $b$-subhomogeneous, and in particular FDI. The universal property of $\cA_{a/b}$
  yields a representation $\sigma: \cA_{a/b} \to B(L)$ with
  $\sigma(u_1) = V_1$ and $\sigma(u_2) = V_2$. Let
  \begin{equation*}
    S = \spa \{ u_1^n u_2^m, u_2^{-m} u_1^{-n} : 0 \le n,m \le N \} \subset \cA_{a/b}
  \end{equation*}
  and let
  \begin{equation*}
    \varphi: S \to B(H), \quad a \mapsto P_H \sigma(a) \big|_H.
  \end{equation*}
  Then $\varphi$ is u.c.p.\ and
  \begin{equation*}
    \varphi(u_1^n u_2^m) = P_H V_1^n V_2^m \big|_H = T_1^n T_2^m \quad \text{ for all } 0 \le n,m \le N.
  \end{equation*}
  By Theorem \ref{thm:main}, the u.c.p.\ map $\varphi$ dilates to a finite-dimensional
  representation $\pi: \cA_{a/b} \to B(K)$. Let $U_1 = \pi(u_1)$ and $U_2 = \pi(u_2)$. Then
  $U_1,U_2$ are $q$-commuting unitaries on a finite-dimensional Hilbert space and
  \begin{equation*}
    T_1^n T_2^m = \varphi(u_1^n u_2^m) = P_H \pi(u_1^n u_2^m) \big|_H
    = P_H U_1^n U_2^m \big|_H
  \end{equation*}
  for all $0 \le n,m \le N$. The dimension bound follows from Proposition \ref{prop:quantitative},
  as $\dim(S) \le 2 (N+1)^2 - 1$.
\end{proof}

\begin{rem}
  The rationality assumption in Corollary \ref{cor:q_commuting} is essential. Indeed, if $q = \exp(2 \pi i \theta)$
  with $\theta$ irrational, then there are no $q$-commuting unitaries on a finite-dimensional Hilbert space,
  because the irrational rotation algebra $\cA_\theta$ is simple and infinite-dimensional, see
  \cite[Theorem VI.1.4]{Davidson96}.
  On the other hand, it is easy to construct $q$-commuting contractions on a finite-dimensional
  Hilbert space, for instance
  \begin{equation*}
    T_1 =
    \begin{bmatrix}
      1 & 0 \\ 0 & q
    \end{bmatrix},
    T_2 =
    \begin{bmatrix}
      0 & 1 \\
      0 & 0
    \end{bmatrix}.
  \end{equation*}
  Thus, Corollary \ref{cor:q_commuting} fails without the rationality assumption.

  In other words, the dilation theorem for $q$-commuting contractions has a finite
  dimensional version if and only if $q = \exp(2 \pi i \theta)$ and $\theta$ is rational.
  This fact becomes very transparent on the level of $C^*$-algebras. Rational rotation algebras
  are subhomogeneous, whereas irrational rotation algebras are simple and infinite-dimensional and hence have
  no finite-dimensional representations.
\end{rem}

A similar phenomenon occurs in \cite[Theorem 6.1]{GS19}, where $q$-commuting unitaries
are dilated to $q'$-commuting unitaries.

\bibliographystyle{amsplain}
\bibliography{literature}

\end{document}